\documentclass[12pt,english,reqno]{amsart}
\usepackage{geometry}
\usepackage{amsmath,amssymb,amsthm,bbm,mathtools,comment}
\usepackage{amsfonts}
\usepackage[shortlabels]{enumitem}
\usepackage[pdftex,colorlinks,backref=page,citecolor=blue]{hyperref}
\usepackage[mathscr]{euscript}
\usepackage[usenames,dvipsnames]{color}
\usepackage{tikz,babel,adjustbox}
\usepackage[numbers]{natbib}
\usepackage{graphicx}
\usepackage{caption}
\usepackage{subcaption}
\usepackage{verbatim}
\usepackage{array}
\usepackage[frame,cmtip,arrow,matrix,line,graph,curve]{xy}
\usepackage{graphpap, pstricks}
\usepackage{etoolbox}
\usepackage{pifont}
\usepackage[noabbrev,capitalize]{cleveref}
\usepackage[final]{microtype}
\usepackage{cmtiup}

\hypersetup{pdfpagemode=UseNone,pdfstartview={XYZ null null 1.00}}
\usetikzlibrary{shapes.misc,calc,intersections,patterns,decorations.pathreplacing}
\usetikzlibrary{arrows,arrows.meta,shapes,positioning,decorations.markings}
\tikzstyle arrowstyle=[scale=1]

\allowdisplaybreaks

\makeatletter
\def\@settitle{\begin{center}%
		\bfseries\Large
		\@title
	\end{center}%
}
\patchcmd{\@setauthors}{\MakeUppercase}{\normalsize}{}{}
\makeatother

\theoremstyle{plain}
\newtheorem{theorem}{Theorem}[section]		
\newtheorem{lemma}[theorem]{Lemma}

\newtheorem{proposition}[theorem]{Proposition}

\newtheorem{problem}[theorem]{Problem}

\theoremstyle{remark}

\newcommand{\beq}[1]{\begin{equation}\label{#1}}
\newcommand{\enq}[0]{\end{equation}}

\def\Prob{\mathbb{P}}
\def\E{\mathbb{E}}
\def\N{\mathbb{N}}

\def\R{\mathbb{R}}

\let\emptyset\varnothing

\newcommand{\eps}{\ensuremath{\varepsilon}}

\newcommand{\cH}{\mathcal{H}}
\newcommand{\cG}{\mathcal{G}}
\newcommand{\cF}{\mathcal{F}}
\newcommand{\cT}{\mathcal{T}_n}

\let\originalleft\left
\let\originalright\right
\renewcommand{\left}{\mathopen{}\mathclose\bgroup\originalleft}
\renewcommand{\right}{\aftergroup\egroup\originalright}

\makeatletter
\def\imod#1{\allowbreak\mkern10mu({\operator@font mod}\,\,#1)}
\makeatother

\begin{document}

\title{Down-set thresholds}
\author{Benjamin Gunby}
\address{Department of Mathematics, Rutgers University, Piscataway, NJ 08854, USA}
\email{bg570@connect.rutgers.edu}

\author{Xiaoyu He}
\address{Department of Mathematics, Princeton University, Princeton, NJ 08540, USA}
\email{xiaoyuh@princeton.edu}

\author{Bhargav Narayanan}
\address{Department of Mathematics, Rutgers University, Piscataway, NJ 08854, USA}
\email{narayanan@math.rutgers.edu}

\date{9 December, 2021}
\subjclass[2010]{Primary 05C80; Secondary 60C05}

\maketitle
\begin{abstract}
We elucidate the relationship between the threshold and the expectation-threshold of a down-set. Qualitatively, our main result demonstrates that there exist down-sets with polynomial gaps between their thresholds and expectation-thresholds; in particular, the logarithmic gap predictions of Kahn--Kalai and Talagrand (recently proved by Park--Pham and Frankston--Kahn--Narayanan--Park) about up-sets do not apply to down-sets. Quantitatively, we show that any collection $\cG$ of graphs on $[n]$ that covers the family of all triangle-free graphs on $[n]$ satisfies the inequality $\sum_{G \in \cG} \exp(-\delta e(G^c) / \sqrt{n}) < 1/2$ for some universal $\delta > 0$, and this is essentially best-possible.

\vspace{3mm}
\noindent {\bf Keywords}: Kahn-Kalai conjecture, thresholds, down-sets, triangle-free graphs
\end{abstract}

\section{Introduction}
For a given finite set $X$, a family $\cF\subset 2^X$ is called an {\em up-set} if it is closed under taking supersets, a {\em down-set} if it is closed under taking subsets, and \emph{monotone} if it is either an up-set or a down-set.

Our main result demonstrates that there is no analogue for down-set thresholds of the conjectures of Kahn--Kalai~\citep{KK} (recently resolved by Park and Pham in \citep{ParkPham}) and Talagrand~\citep{Talagrand} (resolved in~\citep{FKNP}) regarding up-set thresholds. The following theorem, our main contribution, answers (in the negative) a question that arose in discussions between Kahn and the third author.

\begin{theorem}\label{thm:informal}
For the down-set $\cT$ of triangle-free graphs on the vertex set $[n]$, the asymptotic growth rates of the expectation-threshold and the fractional expectation-threshold of $\cT$ are both between $\sqrt{1/n}$ and $\sqrt{\log n / n}$,
while the asymptotic growth rate of the threshold of $\cT$ is $1/n$.
\end{theorem}

The problem of locating thresholds has been a central concern in the study of random discrete structures since the seminal work of Erd\H{o}s and R\'enyi~\citep{ER} on threshold phenomena in random graphs. A great deal of work has since gone into locating thresholds of specific properties of interest; see~\citep{BBbook, JLR} and the many references therein, for example. 

Expectation-thresholds were introduced in~\citep{KK} as a comparatively easy (and rather general) way of locating thresholds of up-sets, so Theorem~\ref{thm:informal} comes as a bit of a surprise; the arguments needed to control the expectation-threshold of the aforementioned down-set $\cT$ are somewhat delicate, in stark contrast to what is needed to locate the threshold of $\cT$. We shall explain exactly where Theorem~\ref{thm:informal} fits into the general theory of thresholds once we fill in some background, a task to which we now turn.

For a given finite set $X$ and $p\in [0,1]$, we write $\mu_p$ for the product measure on the power set $2^X$ of $X$ given by
\[\mu_p(S) = p^{|S|}(1-p)^{|X\setminus S|}\]
for all $S \subset X$. For a non-trivial (i.e., not $2^X$ or $\emptyset$) monotone family $\cF\subset 2^X$, the {\em threshold $p_c(\cF)$} of $\cF$ is the unique $p$ for which $\mu_p(\cF)=1/2$; this is well-defined since $\mu_p(\cF)= \sum_{S \in \cF} \mu_p(S)$ is strictly increasing in $p$ when $\cF$ is a non-trivial up-set, and strictly decreasing in $p$ when $\cF$ is a non-trivial down-set.

Following Talagrand~\citep{Talagrand4, Talagrand1}, we say an up-set $\cF \subset 2^X$ is {\em $p$-small} if there is a certificate $\cG \subset 2^X$ such that $\cF\subset  \cG^{\uparrow}$, where $ \cG^{\uparrow}$ is the increasing family generated by $\cG$ (i.e. the family of all sets containing elements of $\cG$), and
\beq{psmall1}
\sum_{S\in \cG}p^{|S|} \leq 1/2;
\enq
in other words, $\cF$ is $p$-small if there is a simple `first-moment' proof of the fact that $p_c(\cF) \ge p$, namely
\[
\mu_p(\cF) \leq \mu_p  ( \cG^{\uparrow}) \le
	\sum_{S\in \cG}p^{|S|} \le 1/2.
\]
Analogously, we say a down-set $\cF \subset 2^X$ is {\em $p$-small} if there is a certificate $\cG \subset 2^X$ such that $\cF\subset \cG^{\downarrow}$, where $\cG^{\downarrow}$ is the decreasing family generated by $\cG$, and
\beq{psmall2}
\sum_{S\in \cG}(1-p)^{|X \setminus S|} \leq 1/2;
\enq
again, this means that there is a simple proof of the fact that $p_c(\cF) \le p$, namely
\[
\mu_p(\cF) \leq \mu_p  ( \cG^{\downarrow}) \le
	\sum_{S\in \cG}(1-p)^{|X \setminus S|} \le 1/2.
\]

Now, we define the {\em expectation-threshold $q(\cF)$} of a monotone $\cF$ as follows: for an up-set $\cF$, this is the \emph{largest} value $p \in [0,1]$ for which $\cF$ is $p$-small, and for a down-set $\cF$, this is the \emph{smallest} value $p \in [0,1]$ for which $\cF$ is $p$-small. 

Viewing the certificates $\cG$ in~\eqref{psmall1} and~\eqref{psmall2} as integral maps from $2^X$ to $\{0,1\}$, the consideration of fractional relaxations of these certificates leads us to the \emph{fractional expectation-threshold} $q_f(\cF)$ of a monotone $\cF$, which is the optimal value $p \in [0,1]$ for which there is a fractional certificate witnessing either the fact that $p_c(\cF) \ge p$ for an up-set $\cF$, or the fact that $p_c(\cF) \le p$ for a down-set $\cF$. Since we shall primarily focus on expectation-thresholds at this stage of the discussion, we defer a careful discussion of these fractional issues to Section~\ref{sec:frac}. Nevertheless, we note that it follows immediately from these definitions that for any up-set $\cF$, we have $q(\cF) \le q_f(\cF) \le p_c(\cF)$, while for any down-set $\cF$, we instead have $p_c(\cF) \le q_f(\cF) \le q(\cF)$.

This paper gets its motivation from~\citep{KK}, where Kahn and Kalai (drawing on a number of important examples in random graph theory) conjectured  that $p_c(\cF) = O(q(\cF)\log |X|)$ for any up-set $\cF \subset 2^X$, a conjecture recently proven by Park and Pham~\cite{ParkPham}. An earlier slight weakening of this result, originally conjectured by Talagrand~\citep{Talagrand}, was proved in~\citep{FKNP}, where it was shown that $p_c(\cF) = O(q_f(\cF)\log |X|)$ for any up-set $\cF \subset 2^X$. Since decreasing properties also appear quite frequently in the study of random discrete structures, it is natural to ask if there are analogues of the results of~\citep{ParkPham, FKNP} for down-sets. Concretely, our primary motivation is the following question that arose from discussions between Kahn and the third author.
\begin{problem}\label{motivation}
Is it true that $q(\cF)/p_c(\cF) \le (\log |X|)^{O(1)}$ for every down-set $\cF \subset 2^X$?
\end{problem}
Of course, down-sets are just complements of up-sets, so one must ask if there is any new content in Problem~\ref{motivation}, or if it is trivially resolved by the existing machinery for up-sets in~\citep{KK, ParkPham, Talagrand, FKNP}. Problem~\ref{motivation} is indeed nontrivial, but this merits a short explanation. The results of~\citep{ParkPham, FKNP} are all only meaningful for `large' up-sets $\cF$ for which $q(\cF) = o(1)$; indeed, if $\cF$ is `small' in the sense of $q(\cF) = \Omega(1)$, then $1 \ge p_c(\cF) \ge q(\cF)= \Omega(1)$, so the Park-Pham theorem for such `small' $\cF$ holds no content. Nevertheless, the following is a natural analogous question for `small' up-sets in the spirit of the Park-Pham theorem.
\begin{problem}\label{motivation2}
Is it true that $q'/p' \le (\log |X|)^{O(1)}$ for every up-set $\cF \subset 2^X$, where $p_c (\cF) = 1 - p'$ and $q(\cF) = 1 - q'$? 
\end{problem}
It is easy to check that Problems~\ref{motivation} and~\ref{motivation2} are equivalent, as every down-set $\cF$  can be turned into an up-set $\cF'=\{X-A: A\in\cF\}$ with $p(\cF')=1-p(\cF)$ and $q(\cF')=1-q(\cF)$. In this way, all our results on down-sets can also be reformulated in the complementary setting of `small' up-sets. In turn, Theorem~\ref{thm:informal} demonstrates that the answer to Problem~\ref{motivation} (and hence Problem~\ref{motivation2} as well) is in the negative: there exist down-sets with polynomial gaps between their thresholds and (fractional) expectation-thresholds.

We now turn to a discussion of the sharpness of Theorem~\ref{thm:informal}, as well as its proof.
For the down-set $\cT$ of triangle-free graphs on the vertex set $[n]$, it is an easy exercise to show that $p_c(\cT) = \Theta(1/n)$; see~\citep{AS}, for instance. Thus, the heart of the matter is to establish that $q(\cT) = \Omega(1/\sqrt{n})$ (along with a similar bound for $q_f(\cT)$). Therefore, writing $e(H)$ for the number of edges in a graph $H$, we restrict ourselves to proving the following.
\begin{theorem}\label{theorem:triangle-free-existence}
There exists a universal $\delta>0$ such that the following holds for all sufficiently large $n \in \N$. If $\cH$ is a collection of graphs on $[n]$, and 
\begin{equation}\label{equation:prob-assumption}
\sum_{H\in\cH}\exp\left({(-\delta e(H))/\sqrt{n}}\right)<\frac{1}{2}, 
\end{equation}
then there is a triangle-free graph $G$ on $[n]$ sharing at least one edge with each $H\in\cH$.
\end{theorem}
A certificate $\cG$ witnessing the expectation-threshold $q(\cT)$ is just a collection of graphs on $[n]$ that together cover all the triangle-free graphs on $[n]$. Taking $\cH$ to be the collection of complements of graphs in $\cG$, we see that Theorem~\ref{theorem:triangle-free-existence} is equivalent to the fact that $q(\cT) = \Omega(1/\sqrt{n})$. As we shall see in Section~\ref{sec:frac}, the proof of Theorem~\ref{theorem:triangle-free-existence}, with very minor alterations, also allows us to prove that $q_f(\cT) = \Omega(1/\sqrt{n})$, but we focus on $q(\cT)$ and Theorem~\ref{theorem:triangle-free-existence} to keep the presentation simple.

Both Theorems~\ref{thm:informal} and~\ref{theorem:triangle-free-existence} are best-possible up to logarithmic factors. As we shall see in Section~\ref{sec:unicov}, the machinery of hypergraph containers~\citep{cont1, cont2} quickly demonstrates that $q(\cT) = O(\log n / \sqrt{n})$, and a slightly more careful calculation using the Ramsey number upper bound of $r(3,k) = O(k^2/\log k)$ of Ajtai--Koml\'{o}s--Szemer\'{e}di~\citep{AKSRamsey} shows that in fact $q(\cT) = O(\sqrt{\log n/ n})$. This in turn demonstrates that the conclusion of Theorem~\ref{theorem:triangle-free-existence} need not hold if the $\sqrt{1/n}$ in the exponent in~\eqref{equation:prob-assumption} is replaced with anything growing faster than $\sqrt{\log n/ n}$.

We also remark that \cref{theorem:triangle-free-existence} immediately reproduces the Ramsey number lower bound of $r(3,k) = \Omega(k^2/\log^2 k)$; recall that $r(3,k)$ is the least integer $n$ for which each triangle-free graph on $n$ vertices has independence number at least $k$. The aforementioned lower bound was proved independently by Erd\H{o}s~\citep{Erdos} and Spencer~\citep{Spencer} (but is of course superseded by Kim's~\citep{KimRamsey} tour de force bound of $r(3,k) = \Omega(k^2/\log k)$, which gives the correct asymptotic growth rate). Indeed, for $C> 0$ large enough, taking $\cH$ to be the collection of all cliques of size $C\sqrt{n} \log n$ on $[n]$, we see that
\begin{align*}
\sum_{H\in \cH} \exp\left(-\delta e(H)n^{-1/2})\right) &= \binom{n}{C\sqrt{n} \log n}  \exp\left(-\delta C^2 \sqrt{n} \log ^2 n / 2\right)\\
&\le\exp\left((C-\delta C^2/2)\sqrt{n}\log^2 n\right)\\
&< \frac{1}{2},
\end{align*}
so \cref{theorem:triangle-free-existence} implies the existence of a triangle-free graph $G$ on the vertex set $[n]$ with no independent set of size $C\sqrt{n} \log n$.

The proof of \cref{theorem:triangle-free-existence} proceeds by generalising of Erd\H{o}s' lower bound for $r(3,k)$. Erd\H{o}s built triangle-free graphs with large independence number by starting with an Erd\H{o}s--R\'enyi random graph $\Gamma = G(n,p)$ with $p \asymp n^{-1/2}$ and picking a maximal triangle-free subgraph $G\subset \Gamma$; he showed that with high probability, such a $G$ has an edge in every set of vertices of size at least $\sqrt{n}\log n$. We shall show, more generally, that for any choice of $\cH$ satisfying (\ref{equation:prob-assumption}), it is likely for a similarly constructed triangle-free $G$ to share an edge with every $H\in \cH$. 

One of the technical tools that we need to execute the above strategy is a large-deviations estimate for the number of triangles completed by \emph{any} fixed set of edges in the random graph. This large-deviations estimate, and the mechanism we use to `preserve independence' in the proof of this estimate (which involves passing from random graphs to random directed graphs and back) might both be of independent interest; the details appear in Sections~\ref{sec:proof} and~\ref{sec:ldproof}.

This paper is organised as follows. The proof of \cref{theorem:triangle-free-existence}, modulo the large-deviations estimate mentioned above, is given in Section~\ref{sec:proof}, and the beef follows in Section~\ref{sec:ldproof} where the requisite large-deviations estimate is established. The fractional analogue of our main result is stated and sketched in Section~\ref{sec:frac}. We present the constructions demonstrating the sharpness of our results in Section~\ref{sec:unicov}. Finally, we close in Section~\ref{sec:conc} with a discussion of some open problems.

\section{Proof of the main result}\label{sec:proof}

Writing $G(n,p)$ for the Erd\H{o}s--R\'enyi random graph, the proof of Theorem~\ref{theorem:triangle-free-existence} follows from a large-deviations estimate for the number of triangles in $G(n,p)$ `closed' by the addition of \emph{any} given set of $m$ edges. To state this lemma, we need a little bit of notation: given a graph $\Gamma$, we write $e(\Gamma)$ for the number of edges of $\Gamma$, $\Delta(\Gamma)$ for the maximum degree of $\Gamma$, and $\Gamma^2$ for the square of $\Gamma$, which is the graph on the same vertex set as $\Gamma$ whose edges are all pairs of vertices with a common neighbour in $\Gamma$. At the heart of our argument is the following fact.

\begin{lemma}\label{lemma:single-H}
There exist universal $\gamma,\eps>0$ such that the following holds for all sufficiently large $m,n \in \N$. Let $H$ be a graph on $[n]$ with $m$ edges. If $\Gamma \sim G(n,p)$ with $p\le\eps n^{-1/2}$, then
\[
\Prob\left(e(H\cap \Gamma^2) \ge 3m/4 \bigwedge \Delta(\Gamma) \le 2pn\right) \le 15 \exp\left(-\gamma mn^{-1/2}\right).
\]
\end{lemma}

We defer the proof of \cref{lemma:single-H} to the next section, but briefly mention why the lemma demands that we work within the event $\{\Delta(\Gamma) < 2pn\}$; by considering the case where $H$ is a star with $m=\Theta(n)$ edges, we see that $\Prob(e(H\cap \Gamma^2) \ge 3m/4) = p^{\Theta(n)}$, thus demonstrating that the bound in \cref{lemma:single-H} does not hold uniformly in $H$ for $\Prob(e(H\cap \Gamma^2) \ge 3m/4)$. We now describe how to deduce \cref{theorem:triangle-free-existence} from \cref{lemma:single-H}.

\begin{proof}[Proof of \cref{theorem:triangle-free-existence}.]
Take $\gamma, \eps>0$ to be the constants in \cref{lemma:single-H}, and put $p=\eps n^{-1/2}$ and $\Gamma= G(n,p)$. For $H\in \cH$, we say that an edge $uv\in E(H)$ is \textit{$H$-good} if $uv \not \in E((\Gamma\setminus H)^2)$, i.e., $uv$ is an edge of $H$ that does not form any triangle $uvw$ where $uw,vw \in E(\Gamma \setminus H)$. 

We claim that if $\Gamma$ contains an $H$-good edge, then every maximal triangle-free subgraph of $\Gamma$ intersects $H$. Indeed, suppose for contradiction that $\Gamma$ contains an $H$-good edge $e$, and that $G$ is a maximal triangle-free subgraph of $\Gamma$ with $G \cap H =\emptyset$. We see that $G \cup \{e\}$ is also triangle-free, since any triangle thereof must contain $e$, but this would violate the $H$-goodness of $e$. This would imply that $G \cup \{e\}$ is also a triangle-free subgraph of $\Gamma$, violating the maximality of $G$.

It remains to show that with positive probability, $\Gamma$ contains an $H$-good edge for every $H \in \cH$. This proves the theorem, as any maximal triangle-free subgraph $G$ of $\Gamma$ then demonstrates the desired result. To this end, we define three types of bad events, and show that they are all unlikely to happen. Let $Z$ be the event that $\Delta(\Gamma) > 2pn$. For $H\in \cH$, let $Y_H$ be the event that $e(H \cap \Gamma^2) \ge 3e(H)/4$, and let $X_H$ be the event that $\Gamma$ contains no $H$-good edges. 

Writing $\overline{E}$ for the complement of an event $E$, we have
\begin{equation}\label{equation:prob-breakdown}
\Prob\left(\bigvee_{H\in \cH} X_H\right) \le \Prob(Z) + \sum_{H \in \cH} \Prob\left(Y_H \wedge \overline{Z}\right) + \sum_{H\in \cH} \Prob\left(X_H \wedge \overline{Y_H}\right),
\end{equation}
since either $Z$ occurs, some $Y_H$ occurs without $Z$, or else some $X_H$ occurs without the corresponding $Y_H$. We now treat the three terms on the right hand side of (\ref{equation:prob-breakdown}) separately. 

For the first, it is easily seen that $\Prob(Z)\rightarrow 0$ as $n\rightarrow \infty$. For the second term, we apply \cref{lemma:single-H} to obtain
\begin{equation}\label{equation:gamma}
\Prob\left(Y_H \wedge \overline{Z}\right) \le 15\exp\left(-\gamma e(H)n^{-1/2}\right).
\end{equation}
For the third and final term in (\ref{equation:prob-breakdown}), note that $\Prob(X_H \wedge \overline{Y_H}) \le \Prob(X_H \,|\, \overline{Y_H})$. We can understand this conditional probability of $X_H$ by exposing the randomness of $\Gamma = G(n,p)$ in two stages. First, we expose the edges of $\Gamma$ outside $H$ and count the number of $H$-good edges in $H$. We then expose the edges of $\Gamma$ in $H$, noting that when doing this, each edge of $H$ has an independent $p$ chance of lying inside $\Gamma$. Thus,
\begin{align}
\Prob\left(X_H \,|\, \overline{Y_H}\right) & \le \Prob\left(X_H \,|\, e(H \setminus (\Gamma\setminus H)^2) \ge e(H)/4\right) \nonumber\\
& \le (1-p)^{e(H)/4} \nonumber\\
& \le \exp\left(-\frac{\eps}{4} e(H) n^{-1/2}\right).\label{equation:epsilon}
\end{align}

Now, we fix $\delta = \min(\eps/4, \gamma)/5$. We know from~\eqref{equation:prob-assumption} that
\[
\sum_{H\in \cH} \exp\left(-\delta e(H)n^{-1/2}\right) < \frac{1}{2};
\]
in particular, each individual term in the sum is less than $1/2$. Thus,
\[
\sum_{H\in \cH} \exp\left(-5\delta e(H)n^{-1/2}\right) < \frac{1}{16}\sum_{H\in \cH} \exp\left(-\delta e(H)n^{-1/2}\right) < \frac{1}{32}.
\]
Putting this together with (\ref{equation:prob-breakdown}), (\ref{equation:gamma}), and (\ref{equation:epsilon}), we find that
\begin{align*}
\Prob\left(\bigvee_{H\in \cH} X_H\right) & \le \sum_{H\in \cH} \left(15\exp\left(-\gamma e(H)n^{-1/2}\right)+ \exp\left(-\frac{\eps}{4} e(H) n^{-1/2}\right)\right) + o(1) \\
& \le 16\sum_{H\in \cH} \exp\left(-5\delta e(H)n^{-1/2}\right) + o(1) < 1
\end{align*}
for all $n\in\N$ sufficiently large, as desired. 

We have shown that with positive probability, $\Gamma$ contains an $H$-good edge for every $H \in \cH$. Thus, with positive probability, any maximal triangle-free subgraph $G$ of $\Gamma$ will share an edge with every $H\in \cH$, as desired.
\end{proof}

\section{Large-deviations for closed triangles}\label{sec:ldproof}
Our proof of \cref{lemma:single-H} needs us, amongst other things, to show that for a fixed graph $H$ on $[n]$ and a random set $U\subset [n]$ containing each vertex with probability $p$, the number of edges in the induced subgraph $H[U]$ is concentrated. In particular, we need the following technical bound on the exponential moment of $e(H[U])$ in the case where $H$ is bipartite.

\begin{lemma}\label{lemma:neighbourhood-concentration}
Let $H$ be a bipartite graph on $[n]$ with $m$ edges, and let $U$ be a random subset of $[n]$ such that each vertex is chosen independently with the same probability $p\in [0,1]$. Let $X$ be the random variable counting the number of edges in $H[U]$, and let $Z$ be the indicator random variable of the event $|U|\leq 5pn$. Then
\[\mathbb{E}\left[\exp\left(\frac{ZX}{5pn}\right)\right]\leq \exp\left({\frac{pm}{n}}\right).\]
\end{lemma}
\begin{proof}
Let $V_L\cup V_R$ be a bipartition of the vertex set of $H$, and let $U_L=U\cap V_L$ and $U_R=U\cap V_R$. We expose the random set $U$ by first exposing the random set $U_L$ and then exposing $U_R$.

For $i\in V_R$, let $d_i$ be the random variable counting the number of neighbours of $i$ in $U_L$, and let $X_i$ be the indicator random variable of the event $i\in U_R$, so that $X=\sum_{i\in V_R} d_iX_i$. Notice that $d_i$ depends only on $U_L$ and not on $U_R$, so conditional on any exposure of $U_L$, the random variables $\{d_i X_i:i\in V_R\}$ are independent. Therefore, taking $Z_L$ to be the indicator of the event $|U_L|\leq 5pn$, and noting that $Z_L$ is $\{0,1\}$-valued and that $Z\leq Z_L$, we get
\begin{align}
    \nonumber \mathbb{E}\left[\exp\left({\frac{ZX}{5pn}}\right)\right] & \leq \mathbb{E}\left[\exp\left({\frac{Z_LX}{5pn}}\right)\right] =\mathbb{E}\left[1-Z_L+Z_L\exp\left({\frac{X}{5pn}}\right)\right] \\ & \nonumber =\mathbb{E}_{U_L}\left[1-Z_L+\mathbb{E}_{U_R}\left[Z_L\exp\left({\frac{X}{5pn}}\right)\right]\right] \\ & \nonumber =\mathbb{E}_{U_L}\left[1-Z_L+Z_L\mathbb{E}_{U_R}\left[\exp\left({\frac{\sum_{i\in V_R} d_iX_i}{5pn}}\right)\right]\right] \\ & \nonumber =\mathbb{E}_{U_L}\left[1-Z_L+Z_L\prod_{i\in V_R}\mathbb{E}_{X_i}\left[\exp\left({\frac{d_iX_i}{5pn}}\right)\right]\right] \\ & \label{rightsidebound}=\mathbb{E}_{U_L}\left[1-Z_L+Z_L\prod_{i\in V_R}\left(1-p+p\exp\left({\frac{d_i}{5pn}}\right)\right)\right],
\end{align}
the last equality holding since $X_i$ is a Bernoulli random variable with $\Prob(X_i = 1) = p$. 

We now state a simple estimate that we will use to control~\eqref{rightsidebound}: for all $0\leq x\leq 1$, we have
\beq{taylor}
e^x-1\leq (e-1)x;
\enq
indeed, equality holds at $x=0$ and $x=1$ and the right hand side is linear, so this bound follows from convexity of $e^x-1$.

Returning to~\eqref{rightsidebound}, if $Z_L=1$, then $d_i\leq |U_L|\leq 5pn$; thus~\eqref{taylor} with $x={d_i}/{5pn}$ implies that either $Z_L=0$ or \[\exp\left({\frac{d_i}{5pn}}\right)-1\leq\frac{(e-1)d_i}{5pn}\leq\frac{d_i}{2pn}.\] 
Therefore, \[
    Z_L\prod_{i\in V_R}\left[1-p+p\exp\left({\frac{d_i}{5pn}}\right)\right]  \leq Z_L\prod_{i\in V_R}\left(1+\frac{d_i}{2n}\right) \leq Z_L\exp\left({\frac{\sum_{i\in V_R} d_i}{2n}}\right),
\]
which when plugged into (\ref{rightsidebound}) gives
\begin{equation}\label{rightsidedone}
    \mathbb{E}\left[\exp\left({\frac{ZX}{5pn}}\right)\right]\leq \mathbb{E}_{U_L}\left[1-Z_L+Z_L\exp\left({\frac{\sum_{i \in V_R} d_i}{2n}}\right)\right]\leq \mathbb{E}_{U_L}\left[\exp\left({\frac{\sum_{i\in V_R} d_i}{2n}}\right)\right].
\end{equation}

Now, for $j\in V_L$, let $a_j=\deg_H(j)$ and let $Y_j$ be the indicator random variable of the event $j\in U_L$, and note that \[\sum_{i \in V_R} d_i=\sum_{j \in V_L} a_jY_j\] since both sides count the number of edges between $U_L$ and $V_R$. As the random variables $\{Y_j: j \in V_L\}$ are independent, we get
\begin{align}
    \nonumber \mathbb{E}_{U_L}\left[\exp\left({\frac{\sum_{i\in V_R} d_i}{2n}}\right)\right] & =\mathbb{E}_{U_L}\left[\exp\left({\frac{\sum_{j\in V_L} a_jY_j}{2n}}\right)\right] \\ & \nonumber =\prod_{j\in V_L}\mathbb{E}_{Y_j}\left[\exp\left({\frac{a_jY_j}{2n}}\right)\right] \\ & \label{leftsidebound}=\prod_{j\in V_L}\left(1-p+p\exp\left({\frac{ a_j}{2n}}\right)\right),
\end{align}
the last equality holding since $Y_j$ is a Bernoulli random variable with $\Prob(Y_i = 1) = p$.  

To finish, we see that since  $a_j\leq |V_R|\leq n$, the bound~\eqref{taylor} with $x={a_j}/{2n}$ gives 
\[\exp\left({\frac{a_j}{2n}}\right)-1\leq \frac{(e-1)a_j}{2n}\leq\frac{a_j}{n}.\] 
Substituting this bound into~\eqref{leftsidebound}, then substituting the result into~\eqref{rightsidedone}, and using the fact that $\sum_{j \in V_L} a_j=m$, we get
\begin{align*}
    \mathbb{E}\left[\exp\left({\frac{ZX}{5pn}}\right)\right] & \leq\prod_{j\in V_L}\left(1-p+p\exp\left({\frac{a_j}{2n}}\right)\right) \\ & \leq \prod_{j\in V_L}\left(1+\frac{p a_j}{n}\right) \\ & \leq \exp\left({\frac{p\sum_{j \in V_L} a_j}{n}}\right) =\exp\left({\frac{pm}{n}}\right),
\end{align*}
completing the proof.
\end{proof}

To prove \cref{lemma:single-H}, it will be more convenient to work with random directed graphs. This model gives us some additional independence that is crucial in proving the directed analogue of \cref{lemma:single-H} below. We shall subsequently show that \cref{lemma:single-H} follows from its directed analogue.

We need a little more notation. We write $\vec{G}(n,p)$ for the random directed graph (or digraph, for short) on $[n]$, where each directed edge $(u,v)$ appears independently with probability $p$, and pairs of anti-parallel edges $(u,v)$ and $(v,u)$ are allowed to be simultaneously present. For a digraph $D$, we write $\Delta(D)$ for its maximum out-degree, and $\widehat{D}$ for the undirected graph whose edges are pairs of distinct vertices in $D$ with a common in-neighbour. The directed analogue of \cref{lemma:single-H} is as follows.

\begin{lemma}\label{lemma:single-H-reduced}
There exist universal $\gamma',\eps' > 0$ such that the following holds for all sufficiently large $m, n \in \N$. Let $H$ be a bipartite graph on $[n]$ with $m$ edges. If $D \sim \vec{G}(n,p)$ with $p\le \eps' n^{-1/2}$, then
\[
\Prob\left(e(H\cap \widehat{D}) \geq m/16 \bigwedge \Delta(D) \leq 5pn-1\right) \le \exp\left({-\gamma' mn^{-1/2}}\right).
\]
\end{lemma}

\begin{proof}
Let $D_\circ$ be the random directed graph generated in the same way as $D$ except that self-loops are also chosen with probability $p$. Then in the natural coupling between $D$ and $D_\circ$ given by removing the self-loops of $D_\circ$, we have both 
\[e(H\cap \widehat{D}_{\circ})\geq e(H\cap \widehat{D})\] and \[\Delta(D_\circ)\leq\Delta(D)+1,\] whence it follows that
\[\Prob\left(e(H\cap \widehat{D}) \geq m/16 \bigwedge \Delta(D) \leq 5pn-1\right)\leq \Prob\left(e(H\cap \widehat{D}_{\circ}) \geq m/16 \bigwedge \Delta(D_\circ) \leq 5pn\right).\]

For each $v\in [n]$, let $U_v$ be the out-neighbourhood of $v$ in $D_\circ$, and let $X_v$ be the random variable counting the number of edges of $H$ contained in $U_v$. Since any edge $e\in E(\widehat{D}_{\circ})$ is contained in $U_v$ for some $v\in [n]$, we have
\[e(H\cap \widehat{D}_{\circ})\leq\sum_{v\in [n]}X_v.\]

Next, let $Z$ be the indicator random variable of the event that $\Delta(D_\circ)\leq 5pn$. For each $v\in [n]$, let $Z_v$ be the indicator random variable of the event that $|U_v|\leq 5pn$. Clearly, $Z\leq Z_v$ for all $v\in[n]$, so it follows that
\begin{align*}
    \Prob\left(e(H\cap \widehat{D}_{\circ}) \geq m/16 \bigwedge \Delta(D_\circ) \leq 5pn\right) & \leq\Prob\left(Z\sum_{v\in [n]}X_v\geq m/16\right) \\ & \leq \Prob\left(\sum_{v\in [n]}Z_vX_v\geq m/16\right).
\end{align*}
Note that the random variables $\{Z_vX_v: v \in [n]\}$ are independent, as $Z_vX_v$ only depends on the out-neighbourhood $U_v$ of $v$, and all of these are independent; this decoupling is why we find it crucial to work with random \emph{directed} graphs. 

Now, $U_v$ is a random subset of $[n]$ with each vertex chosen independently with probability $p$, so by \cref{lemma:neighbourhood-concentration}, we have
\[\mathbb{E}\left[\exp\left({\frac{Z_vX_v}{5pn}}\right)\right]\leq \exp\left({\frac{pm}{n}}\right)\]
for every $v\in [n]$. Therefore, using an exponential moment bound, we get
\begin{align*}
    \Prob\left(\sum_{v\in [n]}Z_vX_v\geq m/16\right) & = \Prob\left(\exp\left({\frac{\sum_{v\in [n]} Z_vX_v}{5pn}}\right)\geq \exp\left({\frac{m}{80pn}}\right)\right) \\ 
    & \leq\frac{\mathbb{E}\left[\exp\left({\frac{\sum_{v \in [n]} Z_vX_v}{5pn}}\right)\right]}{\exp\left({\frac{m}{80pn}}\right)} \\ 
    & =\frac{\prod_{v\in [n]}\mathbb{E}\left[\exp\left({\frac{Z_vX_v}{5pn}}\right)\right]}{\exp\left({\frac{m}{80pn}}\right)} \\ 
    & \leq \exp\left({\frac{m}{pn}\left(p^2n-\frac{1}{80}\right)}\right).
\end{align*}
For $p\le(20\sqrt{n})^{-1}$, say, we have $p^2n-1/80\leq -1/100$, so
\[\Prob\left(\sum_{v\in [n]}Z_vX_v\geq m/16\right)\leq \exp\left({-\frac{m}{100pn}}\right)\le \exp\left({-\frac{m}{5\sqrt{n}}}\right).\]
Thus, \cref{lemma:single-H-reduced} is seen to hold with $\gamma'=1/5$ and $\eps'=1/20$.
\end{proof}

Finally, the proof of \cref{lemma:single-H} follows from  \cref{lemma:single-H-reduced} and Markov's inequality by constructing a suitable coupling.

\begin{proof}[Proof of Lemma~\ref{lemma:single-H}]
Let $H$ be an arbitrary graph on $[n]$ with $m$ edges. Choose a bipartite subgraph $H'$ of $H$ with $m'\ge m/2$ edges. Let $\gamma',\eps'>0$ be the constants in \cref{lemma:single-H-reduced}, suppose $p\le\eps'n^{-1/2}$, and pick $p'$ such that $2p' - (p')^2 = p$. Then $p' \le \eps 'n^{-1/2}$ satisfies the conditions of \cref{lemma:single-H-reduced}, and for $m,n\in\N$ sufficiently large, \cref{lemma:single-H-reduced} applies to $H'$ and $p'$, telling us that if $D \sim \vec{G}(n,p')$, then
\[
\Prob\left(e(H'\cap \widehat{D}) \ge m'/16 \bigwedge \Delta(D) \leq 5p'n-1\right) \le \exp\left({-\gamma' m'n^{-1/2}}\right).
\]

Let $\Gamma\sim G(n,p)$ be an undirected Erd\H{o}s--R\'enyi random graph. By our choices of $p$ and $p'$, we can couple $\Gamma \sim G(n,p)$ and $D\sim \vec{G}(n,p')$ such that two vertices in $\Gamma$ are adjacent if and only if there exists at least one directed edge between them in $D$. In particular, $\Delta(D) \le \Delta(\Gamma)$, so 
\[
\Delta(\Gamma) \le 2pn \implies \Delta(D) \le 5p'n - 1.
\]
Since $H'$ contains at least $m'\ge m/2$ edges of $H$, we also have
\[
e(H \cap \Gamma^2) \ge 3m/4 \implies e(H' \cap \Gamma^2) \ge m'/2.
\]

Now, if an edge lies in $\Gamma^2$, then by the coupling above, it has at least a $1/4$ chance of lying in $\widehat{D}$. Thus, for any fixed graph $\Gamma_0$ on $[n]$ with $e(H'\cap \Gamma_0^2) \ge m'/2$, we have 
\[\E[e(H' \cap \widehat{D}) \,|\, \Gamma = \Gamma_0] \ge m'/8.\] 
Since $e(H' \cap \widehat{D})$ is a random variable supported on $[0, m']$, Markov's inequality applied to $X = m' - e(H' \cap \widehat{D})$ yields
\[
\Prob\left(e(H' \cap \widehat{D}) < m'/16 \,|\, \Gamma = \Gamma_0\right) = \Prob\left(X \ge 15m'/16 \,|\, \Gamma = \Gamma_0\right) \le \frac{7m'/8}{15m'/16} = \frac{14}{15},
\]
implying that for any $\Gamma_0$ with $e(H\cap\Gamma_0^2) \ge 3m/4$, we have
\[
\Prob\left(e(H' \cap \widehat{D}) \ge m'/16 \,|\, \Gamma = \Gamma_0\right) \ge \frac{1}{15}.
\]
Summing this estimate over all possible $\Gamma_0$, it follows that
\begin{align*}
\Prob&\left(e(H\cap \Gamma^2) \ge 3m/4  \bigwedge \Delta(\Gamma) \leq 2pn\right) \\
& \le 15\Prob\left(e(H'\cap \widehat{D}) \ge m'/16 \bigwedge \Delta(D) \leq 5p'n-1\right) \\
& \le 15\exp\left({-\gamma'm'n^{-1/2}}\right) \\
& = 15\exp\left({-\gamma m n^{-1/2}}\right),
\end{align*}
where $\gamma = \gamma'/2$; this completes the proof.
\end{proof}

\section{Fractional relaxations}\label{sec:frac}
We now turn to the fractional analogue of \cref{theorem:triangle-free-existence}; to state this carefully, we need some definitions.

Following Talagrand~\citep{Talagrand4, Talagrand1} once again, we say an up-set $\cF \subset 2^X$ is {\em weakly $p$-small} if there is a certificate $a: 2^X \to \R_{\ge 0}$ such that 
for each $S \in \cF$, we have
\[	\sum_{T\subset S}a(T)\geq 1, \] and
\[
\sum_{T \subset X}a(T)p^{|T|} \leq 1/2;
\]
thus, if $\cF$ is weakly $p$-small, then we have a simple proof, as before, of the fact that $p_c(\cF) \ge p$.
Analogously, we say a down-set $\cF \subset 2^X$ is {\em weakly $p$-small} if there is a certificate $a: 2^X \to \R_{\ge 0}$ such that 
for each $S \in \cF$, we have
\[	\sum_{T\supset S}a(T)\geq 1, \] and
\[
\sum_{T \subset X}a(T)(1-p)^{|X \setminus T|} \leq 1/2;
\]
thus, if $\cF$ is weakly $p$-small, then we again have a simple proof of the fact that $p_c(\cF) \le p$.

Now, we define the {\em fractional expectation-threshold $q_f(\cF)$} of a monotone $\cF$ as follows: for an up-set $\cF$, this is the \emph{largest} value $p \in [0,1]$ for which $\cF$ is weakly $p$-small, and for a down-set $\cF$, this is the \emph{smallest} value $p \in [0,1]$ for which $\cF$ is weakly $p$-small.

The statement that $q_f(\cT)=\Omega(1/\sqrt{n})$ is equivalent to the following modification of \cref{theorem:triangle-free-existence}.
\begin{theorem}\label{theorem:triangle-free-existence-fractional}
There exists a universal $\delta>0$ such that the following holds for all sufficiently large $n \in \N$. If $\cH$ is the collection of all graphs on $[n]$ and $a: \cH \to \R_{\ge 0}$ satisfies
\[\sum_{H\in\cH}a(H)\exp\left({(-\delta e(H))/\sqrt{n}}\right)<\frac{1}{2},\]
then there is a triangle-free graph $G$ on $[n]$ such that
\[\sum_{H\in\cG(G)}a(H) < 1,\]
where $\cG(G) = \{H\in\cH: G \cap H = \emptyset\}$.
\end{theorem}
\begin{proof}
This can be proved in a manner similar to \cref{theorem:triangle-free-existence}. In particular, let $\Gamma=G(n,p)$, let $G$ be any maximal triangle-free subgraph of $\Gamma$, and define $X_H$, $Y_H$, and $Z$ as in the proof of \cref{theorem:triangle-free-existence}. Recall that if $\overline{X_H}$ holds, then $G\cap H\neq\emptyset$. Thus
\begin{align*}
    \mathbb{E}\left[\left.\sum_{H \in \cG(G)}a(H)\,\right|\,\overline Z\right] & = \sum_{H\in\cH}a(H)\Prob\left(G\cap H=\emptyset\,|\,\overline {Z}\right) \\ & \leq\sum_{H\in\cH}a(H)\Prob\left(X_H\,|\,\overline {Z}\right) \\ & =\sum_{H\in\cH}a(H)\frac{\Prob\left(X_H\wedge\overline{Z}\right)}{\Prob\left(\overline{Z}\right)}.
\end{align*}
By the bounds on $\Prob(X_H\wedge\overline{Y_H})$ and $\Prob(Y_H\wedge\overline{Z})$ in the proof of \cref{theorem:triangle-free-existence}, we have
\[\Prob\left(X_H\wedge\overline{Z}\right)\leq 16\exp\left({-ce(H)n^{-1/2}}\right)\]
for some absolute constant $c>0$. Of course, this probability is bounded above by $1$, so since $\Prob(\overline{Z})=1-o(1)$, we get
\[\mathbb{E}\left[\left.\sum_{H \in \cG(G)}a(H)\,\right|\,\overline Z\right]\leq (1+o(1))\sum_{H\in\cH}a(H)\min\left(16\exp\left({-ce(H)n^{-1/2}}\right),1\right).\]
Putting $\delta={c}/{8}$, we then have
\[\min\left(16\exp\left({-ce(H)n^{-1/2}}\right),1\right) \leq \sqrt{2}\exp\left({-\delta e(H)n^{-1/2}}\right),\]
as the right hand side is the geometric mean of one copy of $16 \exp\left({-ce(H)n^{-1/2}}\right)$ and seven copies of $1$. Thus, if 
\[\sum_{H\in\cH}a(H)\exp\left({(-\delta e(H))/\sqrt{n}}\right)<\frac{1}{2},\]
then
\begin{align*}
\mathbb{E}\left[\left.\sum_{H \in \cG(G)}a(H)\,\right|\,\overline Z\right] &\leq \left(\sqrt{2}+o(1)\right)\sum_{H\in\cH}a(H)\exp\left({(-\delta e(H))/\sqrt{n}}\right)\\ &\le\frac{1}{\sqrt{2}}+o(1)<1,
\end{align*}
proving that there is some triangle-free graph $G$ on $[n]$ for which
\[\sum_{H \in \cG(G)}a(H)<1,\]
as required.
\end{proof}

\section{Uniform covers}\label{sec:unicov}
Our goal in this section is to collect together some constructions that yield good upper bounds on the expectation-threshold of $\cT$ (and consequently, its fractional expectation-threshold as well, since $q_f(\cT) \le q(\cT)$).

For $0\leq m\leq\binom{n}{2}-{n^2}/{4}$, let $f(m,n)$ be the smallest integer $k$ such that there exists a family $\cH$ of $k$ graphs on $[n]$ with $e(H)=\binom{n}{2}-m$ for all $H\in\cH$ such that every triangle-free graph is contained in some $H\in\cH$.

Notice that $f(m,n)$ exists for all $m$ in the given range, as all triangle-free graphs have at most ${n^2}/{4}$ edges. The relationship between the expectation-threshold $q(\cT)$ and the function $f$ is as follows.
\begin{proposition}\label{prop:threshold-vs-family}
For all $0\leq m\leq\binom{n}{2}-{n^2}/{4}$, we have $q(\cT) \le \log(2f(m,n))/m$.
\end{proposition}
\begin{proof}
Take a family $\cH$ of $f(m,n)$ graphs, each with $\binom{n}{2}-m$ edges, that covers $\cT$. The probability that $G(n,p)$ is contained in one of these graphs is $(1-p)^m$. Setting $p = q(\cT)$, we must have
\[f(m,n)\left(1-p\right)^m\geq 1/2\]
by the minimality of the expectation-threshold; noting that $(1-p)^m\leq e^{-pm}$, the conclusion follows.
\end{proof}

We now collect together estimates for $f(m,n)$ that hold in various regimes.
\begin{proposition}\label{prop:bounds-on-f}
We have the following bounds on $f(m,n)$ for all sufficiently large $n \in \N$.
\begin{enumerate}
    \item\label{complete}For all $0\leq m\leq\binom{n}{2}-{n^2}/{4}$, we have \[f(m,n)=\exp\left({\Omega\left(\max\left(m/\sqrt{n},\sqrt{m}\right)\right)}\right).\]
    \item\label{ramsey} There exists a universal $c>0$ such that if $m<cn\log n$, then \[f(m,n)\leq \exp\left({2\sqrt{m}\log n}\right).\]
    \item\label{containers} For all $0<c<1/4$, there exists $C = C(c) > 0$ such that \[f\left(cn^2,n\right)\leq \exp\left({Cn^{3/2}\log n}\right).\]
\end{enumerate}
\end{proposition}

Before turning to the proof of these estimates, notice that \cref{prop:threshold-vs-family} and \cref{ramsey} in the above proposition combine to show that
\[q_f(\cT) \le q(\cT) = O\left(\frac{\sqrt{n\log n}\log n}{n\log n}\right)=O\left(\sqrt{\frac{\log n}{n}}\right),\]
as was earlier claimed.

\begin{proof}[Proof of \cref{prop:bounds-on-f}]
We start with \cref{complete}. The first inequality here, i.e., $f(m,n)=\exp({\Omega(m/\sqrt{n})})$, comes from \cref{prop:threshold-vs-family}, and our proof of the fact that $q(\cT) =\Omega(1/\sqrt{n})$. For the other bound, consider a random complete bipartite graph $G$ in which each vertex is independently sent to either partition class of $G$ with probability $1/2$. For any graph $H$ with $e(H)=\binom{n}{2}-m$, let $T$ be a spanning forest of $H^c$ (i.e. the union of spanning trees of each connected component), and note that $e(T)\geq v(T)/2=\Omega(\sqrt{m})$. It is not difficult to see that the edges of $T$ are contained in $G$ with independent probabilities of $1/2$, so
    \[\Prob(G\subset H)\leq\Prob(G\cap T=\emptyset)=2^{-e(T)}=\exp\left({-\Omega\left(\sqrt{m}\right)}\right).\]
Since such a $G$ is always triangle-free, if $\cH$ is a family of graphs, each with $\binom{n}{2}-m$ edges, that covers $\cT$, a union bound yields $|\cH|=\exp({\Omega(\sqrt{m})})$, and the second inequality follows.
    
Next, we turn to \cref{ramsey}. Ajtai, Koml\'os and Szemer\'edi~\citep{AKSRamsey} showed that $r(3,n)=O(n^2/\log n)$, i.e., there is a universal $c>0$ such that $r(3,2\sqrt{m})\leq n$ for all $m<cn\log n$. This implies that we can cover $\cT$ by taking all graphs of the form $K_n-B$ where $B$ is a clique of size $2\sqrt{m}$. There are $\binom{n}{2\sqrt{m}}$ such graphs, and they each have $\binom{2\sqrt{m}}{2}\geq m$ non-edges, so this shows that
    \[f(m,n)\leq\binom{n}{2\sqrt{m}}\leq \exp\left({2\sqrt{m}\log n}\right).\]
    
Finally, for \cref{containers}, a standard application of the machinery of hypergraph containers (as in~\citep{cont2}, for example) shows that for all $\varepsilon>0$, there exists $C>0$ such that for all $n\in \N$, there is a family of $\exp({Cn^{3/2}\log n})$ graphs on $[n]$, each with at most $(\frac{1}{4}+\varepsilon)n^2$ edges, that collectively contain all triangle-free graphs on $[n]$; the claimed bound follows by adding edges to each of these graphs until each of them has exactly $\binom{n}{2}-cn^2$ edges.
\end{proof}

\section{Further Questions}\label{sec:conc}
Our results leave open the problems of determining the asymptotic growth rates of $q(\cT)$ and $q_f(\cT)$. We have shown that both these thresholds are $\Omega(1/\sqrt{n})$ and $O(\sqrt{\log n/ n})$, but getting more precise estimates remains an interesting open problem.
\begin{problem}\label{basicprob}
What are the asymptotic growth rates of $q(\cT)$ and $q_f(\cT)$?
\end{problem}

Towards Problem~\ref{basicprob}, notice that \cref{prop:bounds-on-f} effectively bounds $f(m,n)$ when $m= O(n\log n)$ and $m=\Theta(n^2)$, i.e., up to a logarithmic factor in the exponent, $f(m,n)$ behaves like $\exp({\sqrt{m}})$ when $m = O(n\log n)$, and like $\exp({n^{3/2}})$ when $m=\Theta(n^2)$. This motivates the following question.
\begin{problem}
What is the behaviour of $f(m,n)$ for $n\log n\ll m\ll n^2$? Is the bound $f(m,n)=\exp({\Omega(m/\sqrt{n})})$ sharp in this regime (possibly up to a logarithmic factor in the exponent)?
\end{problem}

Another natural question, whose answer might help resolve Problem~\ref{basicprob}, is as follows.
\begin{problem}
For $m = O(n\log n)$ and $m=\Theta(n^2)$, can we eliminate the logarithmic gaps from our bounds on $f(m,n)$?
\end{problem}

Finally, we remark that our machinery is specialised to the case of triangle-free graphs, and does not prove analogous results for larger cliques; this prompts the following question.

\begin{problem}
For $r\ge 4$ and $\cF_{r,n}$ the down-set of all $K_r$-free graphs on $[n]$, what are the asymptotic growth rates of $q(\cF_{r,n})$ and $q_f(\cF_{r,n})$?
\end{problem}

\section*{Acknowledgements}
We are grateful to Noga Alon, Jacob Fox, Simon Griffiths, Jeff Kahn and Wojciech Samotij for stimulating conversations. The second author was supported by NSF grant DMS-2103154, and the third author was supported by NSF grants CCF-1814409 and DMS-1800521.

\bibliographystyle{amsplain}
\bibliography{downset_thresh}

\providecommand{\bysame}{\leavevmode\hbox to3em{\hrulefill}\thinspace}
\providecommand{\MR}{\relax\ifhmode\unskip\space\fi MR }
\providecommand{\MRhref}[2]{%
  \href{http://www.ams.org/mathscinet-getitem?mr=#1}{#2}
}
\providecommand{\href}[2]{#2}
\begin{thebibliography}{10}

\bibitem{AKSRamsey}
M.~Ajtai, J.~Koml\'{o}s, and E.~Szemer\'{e}di, \emph{A note on {R}amsey
  numbers}, J. Combin. Theory Ser. A \textbf{29} (1980), 354--360.

\bibitem{AS}
N.~Alon and J.~H. Spencer, \emph{The probabilistic method}, fourth ed., Wiley
  Series in Discrete Mathematics and Optimization, John Wiley \& Sons, Inc.,
  Hoboken, NJ, 2016.

\bibitem{cont1}
J.~Balogh, R.~Morris, and W.~Samotij, \emph{Independent sets in hypergraphs},
  J. Amer. Math. Soc. \textbf{28} (2015), 669--709.

\bibitem{BBbook}
B.~Bollob\'{a}s, \emph{Random graphs}, Cambridge Studies in Advanced
  Mathematics, vol.~73, Cambridge University Press, Cambridge, 2001.

\bibitem{Erdos}
P.~Erd\H{o}s, \emph{Graph theory and probability ii}, Canad. J. Math.
  \textbf{13} (1961), 346--352.

\bibitem{ER}
P.~Erd\H{o}s and A.~R\'{e}nyi, \emph{On the evolution of random graphs}, Magyar
  Tud. Akad. Mat. Kutat\'{o} Int. K\"{o}zl. \textbf{5} (1960), 17--61.

\bibitem{FKNP}
K.~Frankston, J.~Kahn, B.~Narayanan, and J.~Park, \emph{Thresholds versus
  fractional expectation-thresholds}, Ann. of Math. \textbf{194} (2021),
  475--495.

\bibitem{JLR}
S.~Janson, T.~\L{}uczak, and A.~Rucinski, \emph{Random graphs},
  Wiley-Interscience Series in Discrete Mathematics and Optimization,
  Wiley-Interscience, New York, 2000.

\bibitem{KK}
J.~Kahn and G.~Kalai, \emph{Thresholds and expectation thresholds}, Combin.
  Probab. Comput. \textbf{16} (2007), 495--502.

\bibitem{KimRamsey}
J.~H. Kim, \emph{The {R}amsey number {$R(3,t)$} has order of magnitude
  {$t^2/\log t$}}, Random Structures Algorithms \textbf{7} (1995), 173--207.

\bibitem{ParkPham}
J.~Park and H.T. Pham, \emph{A proof of the {K}ahn-{K}alai conjecture},
  Preprint, \url{https://arxiv.org/abs/2203.17207}.

\bibitem{cont2}
D.~Saxton and A.~Thomason, \emph{Hypergraph containers}, Invent. Math.
  \textbf{201} (2015), 925--992.

\bibitem{Spencer}
J.~Spencer, \emph{Ramsey's theorem -- a new lower bound}, J. Combin. Theory
  Ser. A \textbf{18} (1975), 108--115.

\bibitem{Talagrand4}
M.~Talagrand, \emph{Are all sets of positive measure essentially convex?},
  Geometric aspects of functional analysis ({I}srael, 1992--1994), Oper. Theory
  Adv. Appl., vol.~77, Birkh\"{a}user, Basel, 1995, pp.~295--310.

\bibitem{Talagrand1}
\bysame, \emph{The generic chaining}, Springer Monographs in Mathematics,
  Springer-Verlag, Berlin, 2005.

\bibitem{Talagrand}
\bysame, \emph{Are many small sets explicitly small?}, {P}roceedings of the
  2010 {ACM} {I}nternational {S}ymposium on {T}heory of {C}omputing, 2010,
  pp.~13--35.

\end{thebibliography}
\end{document}